\documentclass[11pt]{amsart}
\usepackage{amssymb}
\usepackage{mathrsfs}
\theoremstyle{plain}
\newtheorem{prop}{Proposition}
\newtheorem{thm}[prop]{Theorem}
\newtheorem*{thm*}{Theorem}

\newtheorem*{addendum*}{Addendum}

\newtheorem*{convention*}{Convention}
\theoremstyle{definition}
\newtheorem*{defn*}{Definition}

\newtheorem*{scholium*}{Scholium}
\theoremstyle{remark}

\newtheorem*{example*}{Example}
\numberwithin{equation}{section}

\newcommand{\CC}{\mathbf{C}}
\newcommand{\FF}{\mathbf{F}}
\newcommand{\NN}{\mathbf{N}}

\newcommand{\ZZ}{\mathbf{Z}}

\newcommand{\sL}{\mathscr{L}}
\newcommand{\cH}{\mathscr{H}}
\newcommand{\cR}{\mathscr{R}}
\newcommand{\hb}{\mathrm{H}_\mathrm{b}}

\newcommand{\se}{\subseteq}
\newcommand{\lra}{\longrightarrow}

\newcommand{\vnotimes}{\mathrel{\bar{\otimes}}}

\begin{document}
\title[The Dixmier problem, lamplighters and Burnside groups]{The Dixmier problem, lamplighters and\\ Burnside groups}
\author[N. Monod]{Nicolas Monod$^\ddagger$}
\address{N.M. --- EPFL, Switzerland}
\thanks{$^\ddagger$Supported in part by the Swiss National Science Foundation}
\author[N. Ozawa]{Narutaka Ozawa$^*$}
\address{N.O. --- The University of Tokyo, Japan}
\thanks{*Supported in part by the Japan Society for the Promotion of Science}
\renewcommand{\subjclassname}{\textup{2000} Mathematics Subject Classification}
%
\begin{abstract}
J.~Dixmier asked in 1950 whether every non-amenable group admits uniformly bounded representations that cannot be unitarised.
We provide such representations upon passing to extensions by abelian groups. This gives a new
characterisation of amenability. Furthermore, we deduce that certain Burnside groups are non-unitarisable, answering a
question raised by G.~Pisier.
\end{abstract}
\maketitle

\section{Introduction}

A group $G$ is said to be \emph{unitarisable} if every uniformly bounded representation
$\pi$ of $G$ on a Hilbert space $\cH$ is unitarisable, \emph{i.e.}\ there is an invertible operator
$S$ on $\cH$ such that $S\pi(\,\cdot\,)S^{-1}$ is a unitary representation.
Dixmier~\cite{Dixmier} proved that all amenable groups are unitarisable and
asked whether unitarisability characterises amenability.
Since unitarisability passes to subgroups and
non-commutative free groups are not unitarisable,
every group containing a non-commutative free group is non-unitarisable.
For these facts and more background, we refer to Pisier~\cite{Pisier, Pisier_survey}.

Recently, a criterion was discovered~\cite{Epstein-Monod} that lead to examples
without free subgroups (see~\cite{Osin,Epstein-Monod}). We shall improve a strategy proposed in~\cite{MonodICM}
in order to apply ergodic methods to the problem.

\medskip

\begin{flushright}
\itshape
Now are our browes bound with Victorious Wreathes\footnote{Shakespeare, \emph{Richard III}, 1:1 (we quote from the 1623 \emph{First Folio}).}
\end{flushright}
Let $G$ and $A$ be groups. Recall that the associated (restricted) \emph{wreath product}, or \emph{lamplighter group}, is the group
\[
A\wr G\ =\ {\textstyle \bigoplus_G A} \rtimes G,
\]
wherein $\bigoplus_G A$ is the restricted product indexed by $G$ upon which $G$ acts by permutation.
We shall be interested in the case where $A$ and hence also $\bigoplus_G A$ is abelian.

\begin{thm}\label{thm:wreath}
For any group $G$, the following assertions are equivalent.

\begin{itemize}
\item[(i)] The group $G$ is amenable.

\item[(ii)] The wreath product $A \wr G$ is unitarisable for all abelian groups $A$.

\item[(iii)] The wreath product $A \wr G$ is unitarisable for some infinite abelian group $A$.
\end{itemize}
\end{thm}

The above theorem leads to a partial answer to a question of G.~Pisier,
namely whether free Burnside groups are unitarisable
(see \emph{e.g.}~\cite{Pisier_survey}).

\begin{thm}\label{thm:Burnside}
Let $m,n,p$ be integers with $m,n\geq 2$, $p\geq 665$ and $n,p$ odd. Then the free Burnside group $B(m, np)$ of exponent $np$ with $m$ generators is non-unitarisable.
\end{thm}

\subsection*{Acknowledgements}
The essential part of this work was done during the authors'
stay at the Institute of Mathematical Sciences in Chennai.
The authors would like to thank Professor V.~S.~Sunder and
IMSc for their very kind hospitality.

\section{Proofs}
Let $G$ be a group and $(\pi,\cH)$ be a unitary representation of $G$.
We write $\sL(\cH)$ for the algebra of bounded operators of $\cH$.
A map $D\colon G \to \sL(\cH)$ is called a \emph{derivation}
if it satisfies the Leibniz rule $D(gh) = D(g)\pi(h) + \pi(g)D(h)$, or equivalently if
the map $\pi_D$ defined by
\[
\pi_D(g) = \begin{pmatrix}\pi(g) & D(g) \\ 0 & \pi(g)\end{pmatrix}
\in \sL(\cH\oplus\cH)
\]
is a group homomorphism.
In that case, $\pi_D$ is a uniformly bounded representation
if and only if $D$ is a bounded derivation.
Moreover, $\pi_D$ is unitarisable if and only if $D$ is inner,
\emph{i.e.}\ there is $T \in \sL(\cH)$ such that $D(g) = \pi(g)T - T\pi(g)$.
(See Lemma 4.5 in~\cite{Pisier} for a proof of this fact.)
To set up a cohomological framework for studying this problem,
we will view $\sL(\cH)$ as a coefficient $G$-module whose $G$-action
is given by the conjugation $g\cdot T = \pi(g)T\pi(g)^*$.
Then, the space of bounded derivations modulo inner derivations is
canonically isomorphic to the first bounded cohomology group $\hb^1(G, \sL(\cH))$.
Hence, to prove non-unitarisability of $G$, it suffices to produce a
unitary $G$-representation $(\pi,\cH)$ for which $\hb^1(G, \sL(\cH)) \neq 0$.

\medskip

We now undertake the proof of Theorem~\ref{thm:wreath}.
It suffices to show that if $A$ is infinite abelian and $G$ is non-amenable,
then the wreath product $H=A\wr G$ is non-unitarisable.

We can and shall assume that $A$ and $G$ are countable. Indeed, since amenability is preserved under direct limits,
$G$ contains some countable non-amenable group $G_0$. Further, $A$ contains an infinite countable
$G_0$-invariant subgroup $A_0$ and $A_0\wr G_0$ is a subgroup of $A\wr G$.
Thus our claim follows since unitarisability passes to subgroups.

\medskip

Let $\FF$ be a countable non-commutative free group. The proof relies on the following two facts.
(1)~$\hb^1(\FF, \sL(\ell_2\FF)) \neq 0$, see the proof of Theorem~2.7* in~\cite{Pisier}.
(2)~Every non-amenable countable group admits a free type~$\mathrm{II}_1$ action whose orbits
contain the orbits of a free $\FF$-action (\cite{Gaboriau-Lyons}), as described below.
The strategy of the proof is to induce $\hb^1(\FF, \sL(\ell_2\FF))$
through this ``randembedding'' in the sense of~\cite{MonodICM}.

We henceforth consider a non-amenable countable group $G$ and the corresponding Bernoulli shift action
on the compact metrisable product space $X=[0,1]^G$ endowed with the product of the Lebesgue measures.
Gaboriau and Lyons prove in~\cite{Gaboriau-Lyons} that the resulting equivalence relation $\cR\se X\times X$
contains the equivalence relation of some free measure-preserving $\FF$-action upon $X$.
In particular, we have commuting $G$-\ and $\FF$-actions on $\cR$ given by the action on the first, respectively
the second coordinate. These actions preserve the $\sigma$-finite measure on $\cR$ provided by integrating over $X$
the counting measure on orbits. Each of these actions admits a fundamental domain; let $Y\se\cR$ be a fundamental
domain for $\FF$. We may now forget the orbit equivalence relation and view $\cR$ just as a standard measure space
with a measure-preserving $G\times\FF$-action such that $G$ admits a fundamental domain $X$ of finite measure
and $\FF$ admits a fundamental domain $Y$.
We identify $\cR$ with $Y \times \FF$ in such a way that $t^{-1}y\in\cR$ corresponds to $(y,t)\in Y \times \FF$.
Then, $s\in\FF$ acts on $Y \times \FF$ by $s(y,t)=(y,ts^{-1})$ and
$g\in G$ acts by $g(y,t)=(g\cdot y,\alpha(g,y)t)$, where $g\cdot y\in Y$ is the (essentially)
unique element in $\FF gy\cap Y\subset\cR$ and $\alpha(g,y)\in\FF$ is the (essentially)
unique element such that $\alpha(g,y)gy=g\cdot y$. It follows that $\alpha$ satisfies the
cocycle relation $\alpha(gh,y)=\alpha(g,h\cdot y)\alpha(h,y)$.

\medskip

We now consider any countable infinite abelian group $A$. We claim that $A$ has a representation
into the unitaries of the von Neumann algebra $L^\infty(Y)$ whose image generates $L^\infty(Y)$
as a von Neumann algebra. By construction, $Y$ is a standard Borel space with a $\sigma$-finite
non-atomic measure. Furthermore, as far as the present claim is concerned, we may temporarily
assume this measure finite since only its measure class is of relevance. Since $A$ is countably infinite,
its Pontryagin dual $\widehat A$ (for $A$ endowed with the discrete topology) is a non-discrete
compact metrisable group. In other words, we have reduced to the case where we may assume that $Y$ is $\widehat A$
endowed with a Haar measure. Fourier transform establishes an isomorphism between $L^\infty(\widehat A)$ and the group
von Neumann algebra $L(A)\se \sL(\ell_2A)$, which is by definition generated by the
unitary regular representation of $A$; this proves the claim.

\smallskip

Returning to the main argument, we view $A$ in the unitary group of
$L^\infty(Y)\cong L^\infty(Y)\otimes\CC 1_{\FF}\subset L^\infty(\cR)$.
Since $A$ and $gAg^{-1}\subset L^\infty(Y)$ commute,
this gives rise to a unitary representation of $H=A\wr G$ on $L^2(\cR)$.
We will prove that $\hb^1(H, \sL(L^2(\cR))) \neq 0$.

\medskip

We write $N=\bigoplus_G A$. Since $N$ is amenable and $\sL(L^2(\cR))$ is a dual module,
a weak-$*$ averaging argument shows that there is a canonical isomorphism
\[
\hb^*(H, \sL(L^2(\cR)))\ \cong\ \hb^*(G, \sL(L^2(\cR))^N)
\]
(see Corollary~7.5.10 in~\cite{Monod}).
With the identification $\cR = Y\times \FF$, one has
\[
\sL(L^2(\cR))^N\ =\ N'\cap \sL(L^2(\cR))\ =\ L^\infty(Y) \vnotimes \sL(\ell_2\FF)
\ \cong\ L^\infty(Y, \sL(\ell_2\FF))
\]
(see Theorem~IV.5.9 in~\cite{TakesakiI}).
Keeping track of the $G$-representation, one sees that $g\in G$ acts on $L^\infty(Y, \sL(\ell_2\FF))$
by $(g\cdot f)(y)=\tau_{\alpha(g,g^{-1}\cdot y)}(f(g^{-1}\cdot y))$,
where $\tau$ denotes the $\FF$-action on $\sL(\ell_2\FF)$.
For ease of notation, we denote the coefficient $\FF$-module $\sL(\ell_2\FF)$ by $V$.
Then, one further has a $G$-isomorphism
\[
L^\infty(Y, V)\ \cong\ L^\infty(\cR, V)^{\FF},
\]
where $f \in L^\infty(Y, V)$ corresponds to $\tilde{f} \in L^\infty(\cR, V)^\FF$
defined by $\tilde{f}(y,t)=\tau_t^{-1}(f(y))$.
Now, $\FF$ acts on $L^\infty(\cR, V)$ by $(s\cdot F)(z)=\tau_s(F(s^{-1}z))$ and
$G$ acts by $(g\cdot F)(z)=F(g^{-1}z)$.
Since both the $\FF$-action and the $G$-action on $\cR$ admit a fundamental domain,
Proposition~4.6 in~\cite{Monod-Shalom} implies that
\[
\hb^*(G,L^\infty(\cR, V)^{\FF})\ \cong\ \hb^*(\FF,L^\infty(\cR, V)^{G})
\ \cong\ \hb^*(\FF,L^\infty(X, V)).
\]
(See also Proposition~5.8 in~\cite{MonodICM}.)
Since $X = \cR / G$ has a finite $\FF$-invariant measure, the inclusion
$V\hookrightarrow L^\infty(X,V)$ has a $G$-equivariant left inverse.
It follows that the corresponding morphism
\[
\hb^*(\FF, V) \lra  \hb^*(\FF,L^\infty(X, V))
\]
is an injection. Therefore, putting all identifications together,
we conclude that there are injections
\[
\hb^*(\FF, \sL(\ell_2\FF)) \lra \hb^*(H, \sL(L^2(\cR)))
\]
in all degrees. Since $\hb^1(\FF, \sL(\ell_2\FF))\neq0$, this completes the proof.\qed

\medskip

Analysing the proof at the level of derivations, the above injection maps
$D\colon \FF \to \sL(\ell_2\FF)$ to $\tilde{D}\colon H \to \sL(L^2(Y,\ell_2\FF))$
defined by
\[
(\tilde{D}(ag)\xi)(y) = a(y)D(\alpha(g,g^{-1}\cdot y))\xi(g^{-1}\cdot y),
\]
where $a\in N$ is viewed as an element of $L^\infty(Y)$, $g\in G$ and $\xi\in L^2(Y,\ell_2\FF)$.

\begin{proof}[Proof of Theorem~\ref{thm:Burnside}]
By a theorem of Adyan~\cite{Adyan83}, the free Burnside group $G=B(2, p)$ is non-amenable.
Therefore, Theorem~\ref{thm:wreath} implies that $(\bigoplus_{\NN}\ZZ/n\ZZ) \wr G$ is non-unitarisable.
Notice that this wreath product is a countably generated group of exponent $np$.
Therefore, by the universal property of free Burnside groups, it is a quotient of $B(\aleph_0, np)$.
In particular, the latter is non-unitarisable. It was shown by \v{S}irvanjan~\cite{Shirvanjan}
that $B(\aleph_0, np)$ embeds into $B(2, np)$ which is therefore also non-unitarisable.
Finally, each $B(m, np)$ surjects onto $B(2, np)$ as long as $m\geq 2$, concluding the proof.
\end{proof}

\bibliographystyle{amsalpha}

\begin{thebibliography}{Mon06}
%
\bibitem[Ady82]{Adyan83}
S. I. Adyan,
Random walks on free periodic groups.
\emph{Izv. Akad. Nauk SSSR Ser. Mat.} \textbf{46} (1982), 1139--1149, 1343.
%
\bibitem[Dix50]{Dixmier}
J. Dixmier,
Les moyennes invariantes dans les semi-groups et leurs applications.
\emph{Acta Sci. Math. Szeged} \textbf{12} (1950), 213--227.
%
\bibitem[EMxx]{Epstein-Monod}
I. Epstein and N. Monod,
Non-unitarisable representations and random forests.
\emph{Preprint.} arXiv:0811.3422
%
\bibitem[GLxx]{Gaboriau-Lyons}
D. Gaboriau and R. Lyons,
A measurable-group-theoretic solution to von Neumann's problem.
\emph{Preprint.} arXiv:0711.1643
%
\bibitem[Mon01]{Monod}
N. Monod,
\emph{Continuous bounded cohomology of locally compact groups.}
Lecture Notes in Mathematics, 1758. Springer-Verlag, Berlin, 2001.
%
\bibitem[Mon06]{MonodICM}
N. Monod,
An invitation to bounded cohomology.
\emph{International Congress of Mathematicians. Vol. II,}
1183--1211, Eur. Math. Soc., Z\"urich, 2006.
%
\bibitem[MS06]{Monod-Shalom}
N. Monod and Y. Shalom,
Orbit equivalence rigidity and bounded cohomology.
\emph{Ann. of Math. (2)} \textbf{164} (2006), 825--878.
%
\bibitem[Osixx]{Osin}
D. Osin,
$L^2$-Betti numbers and non-unitarizable groups without free subgroups.
\emph{Preprint.} arXiv:0812.2093
%
\bibitem[Pis01]{Pisier}
G. Pisier,
\emph{Similarity problems and completely bounded maps.}
Second, expanded edition. Includes the solution to ``The Halmos problem''.
Lecture Notes in Mathematics, 1618. Springer-Verlag, Berlin, 2001.
%
\bibitem[Pis05]{Pisier_survey}
G. Pisier,
Are unitarizable groups amenable?
\emph{Infinite groups: geometric, combinatorial and dynamical aspects,}
323--362, Progr. Math., 248, Birkh\"auser, Basel, 2005.
%
\bibitem[\v Sir76]{Shirvanjan}
V. L. \v Sirvanjan,
Imbedding of the group $B(\infty,n)$ in the group $B(2,n)$.
\emph{Izv. Akad. Nauk SSSR Ser. Mat.} \textbf{40} (1976), 190--208, 223.
%
\bibitem[Tak02]{TakesakiI}
M. Takesaki,
\emph{Theory of operator algebras. I.}
Reprint of the first (1979) edition. Encyclopaedia of Mathematical Sciences, 124.
Operator Algebras and Non-commutative Geometry, 5. Springer-Verlag, Berlin, 2002.
%
\end{thebibliography}

\end{document}